\documentclass[11pt]{amsart}
\usepackage{amsmath,amssymb,amsthm}
\usepackage[latin1]{inputenc}
\usepackage{version,tabularx,multicol}
\usepackage{graphicx,float}

\newcommand{\reff}[1]{(\ref{#1})}

\theoremstyle{plain}
\newtheorem{theo}{Theorem}[section]
\newtheorem{theorem}[theo]{Theorem}

\newtheorem{cor}[theo]{Corollary}
\newtheorem{prop}[theo]{Proposition}

\newtheorem{lem}[theo]{Lemma}
\newtheorem{defi}[theo]{Definition}
\theoremstyle{remark}
\newtheorem{rem}[theo]{Remark}

\newcommand{\cb}{{\mathcal B}}
\newcommand{\cc}{{\mathcal C}}

\newcommand{\ce}{{\mathcal E}}

\newcommand{\cg}{{\mathcal G}}
\newcommand{\ch}{{\mathcal H}}

\newcommand{\cn}{{\mathcal N}}

\newcommand{\cs}{{\mathcal S}}
\newcommand{\Tau}{{\mathcal T}}

\newcommand{\D}{{\mathbb D}}
\newcommand{\E}{{\mathbb E}}

\newcommand{\G}{\mathbb{G}}
\newcommand{\N}{{\mathbb N}}
\renewcommand{\P}{{\mathbb P}}

\newcommand{\R}{{\mathbb R}}

\newcommand{\T}{\mathbb{T}}

\newcommand{\ind}{{\bf 1}}
\newcommand{\pol}{C_{{\rm pol}}(\R)}

\newcommand{\inv}[1]{\mathop{\frac{1}{ #1}}\nolimits}
\newcommand{\expp}[1]{\mathop {\mathrm{e}^{ #1}}}

\begin{document}

\title[Aging]{Detection of cellular aging in a  Galton-Watson process}

\date{\today}
\author{Jean-François Delmas}
\address{
Jean-Fran\c cois Delmas,
CERMICS, Univ.  Paris-Est, 6-8
av. Blaise Pascal, 
  Champs-sur-Marne, 77455 Marne La Vallée, France.}
\email{delmas@cermics.enpc.fr}
\author{Laurence Marsalle}
\address{Laurence Marsalle, 
Lab. P. Painlevé, CNRS UMR 8524, Bât. M2, 
Univ. Lille 1, Cité Scientifique, 59655 Villeneuve d'Ascq Cedex, France.
}
\email{Laurence.Marsalle@univ-lille1.fr}

\begin{abstract}

  We consider the bifurcating Markov  chain model introduced by Guyon to
  detect  cellular aging  from cell  lineage. To  take into  account the
  possibility  for a  cell to  die, we  use an  underlying Galton-Watson
  process to  describe the  evolution of the  cell lineage.  We  give in
  this more general framework a  weak law of large number, an invariance
  principle  and  thus fluctuation  results  for  the  average over  one
  generation or  up to  one generation. We  also prove  the fluctuations
  over  each generation  are independent.  Then we  present  the natural
  modifications of the tests given  by Guyon in cellular aging detection
  within the particular case of the auto-regressive model.
\end{abstract}



\maketitle

\section{Introduction}
This   work  is  motivated   by  experiments   done  by   biologists  on
\textit{Escherichia  coli}, see  Stewart  and al.   \cite{smpt:adormsd}.
\textit{E. coli} is a rod-shaped single celled organism which reproduces
by dividing  in the middle. It produces  a new end per  progeny cell. We
shall  call this new  end the  new pole  whereas the  other end  will be
called the old  pole. The age of a  cell is given by the age  of its old
pole (i.e.  the number of generations in the past of the cell before the
old pole  was produced).   Notice that at  each generation a  cell gives
birth to 2 cells  which have a new pole and one of  the two cells has an
old pole  of age 1  (which corresponds to  the new pole of  its mother),
while  the  other has  an  old  pole with  age  larger  than one  (which
corresponds to the old pole of its mother). The former is called the new
pole daughter and the latter  the old pole daughter.  Experimental data,
see \cite{smpt:adormsd},  suggest strongly that  the growth rate  of the
new pole  daughter is significantly larger  than the growth  rate of the
old   pole   daughter.  For   asymmetric   aging   see  also   Ackermann
\cite{asj:sbad} for  an other case  of asymmetric division,  and Lindner
and al.  \cite{lmdst:aspaacar} or Ackermann and  al. \cite{acbd:oeoa} on
asymmetric damage repartition.

Guyon   \cite{g:ltbmcadca}   studied   a  mathematical   model,   called
bifurcation  Markov chain  (BMC), of  an  asymmetric Markov  chain on  a
regular  binary tree.   This  model allows  to  represent an  asymmetric
repartition for  example of the growth  rate of a cell  between new pole
and old pole daughters. Using this model, Guyon provides tests to detect
a  difference of the  growth rate  between new  pole and  old pole  on a
single  experimental data set,  whereas in  \cite{smpt:adormsd} averages
over  many  experimental  data sets  have  to  be  done to  detect  this
difference. In  the BMC model, cells  are assumed to never  die (a death
corresponds to  no more  division).  Indeed few  death appear  in normal
nutriment saturated  conditions.  However, under  stress condition, dead
cells  can  represent a  significant  part  of  the population.   It  is
therefore natural  to take  this random effect  into account by  using a
Galton-Watson  (GW)  process.   Our  purpose  is to  study  a  model  of
bifurcating Markov chains  on a Galton Watson tree  instead of a regular
tree.   Notice that  inferences  on symmetric  bifurcating processes  on
regular trees have been studied, see the survey of Hwang, Basawa and Yea
\cite{hby:lanbaprai}  and   the  seminal  work  of   Cowan  and  Staudte
\cite{cs:bamcls}.   We also  learned  of a  recent  independent work  on
inferences for asymmetric auto-regressive models by Bercu, De Saporta and
Gégout-Petit  \cite{bdsgp:aabapm}.  Other models  on  cell lineage  with
differentiation  have   been  investigated,  see   for  example  Bansaye
\cite{b:ppdckbmr,b:ccbprei}   on  parasite   infection  and   Evans  and
Steinsaltz   \cite{es:dsfigrsm}   on   asymptotic  models   relying   on
super-Brownian motion.

\subsection{The statistical  model} \label{ssec:stat} In  order to study
the behavior of the growth  rate of cells in \cite{smpt:adormsd}, we set
some notations:  we index  the genealogical tree  by the  regular binary
tree $\T=\{\emptyset \} \cup \bigcup_{k \in \N^*}\{0,1\}^k$; $\emptyset$
is the label of the ancestor and  if $i$ denotes a cell, let $i0$ denote
the  new pole  progeny cell,  and $i1$  the old  pole progeny  cell. The
growth rate  of cell $i$ is $X_i$.   When the mother gives  birth to two
cells among which a unique one  divides, we consider that the cell which
doesn't divide doesn't grow.  We work with the following model (see
Section \ref{par:bmc} for a more general model of BMC on GW tree):
\begin{itemize}
\item With probability $p_{1,0}$, $i$  gives birth to two cells $i0$ and
  $i1$  which  will both  divide.  The  growth  rates of  the  daughters
  $X_{i0}$  and $X_{i1}$  are  then  linked to  the  mother's one  $X_i$
  through the following auto-regressive equations
\begin{equation}
   \label{eq:normal}
\left\{
\begin{array}{rcl}
 X_{i0} & = & \alpha_0 X_i + \beta_0 +\varepsilon_{i0} \\
X_{i1} & = & \alpha_1 X_i + \beta_1 +\varepsilon_{i1}, 
\end{array}
\right.
\end{equation}
where $\alpha_0$, $\alpha_1 \in (-1,1)$, $\beta_0$, $\beta_1 \in \R$ and
$((\varepsilon_{i0},  \varepsilon_{i1}), i  \in  \T)$ is  a sequence  of
independent   centered bi-variate  Gaussian random  variables,  with covariance
matrix
$$ \sigma^2 \left( 
\begin{array}{cc}
 1 & \rho \\
\rho & 1
\end{array}
\right) , \quad \sigma^2 > 0, \quad \rho \in (-1,1).$$

\item With probability $p_0$, only the new pole $i0$ divides. Its growth
  rate $X_{i0}$ is  linked to its mother's one $X_i$ through the
  relation 
\begin{equation}
   \label{eq:dead-0}
X_{i0}  = \alpha_0' X_i + \beta_0' +\varepsilon_{i0}',
\end{equation}
where $\alpha_0' \in (-1,1)$, $\beta_0' \in \R$ and $(\varepsilon_{i0}',
i \in \T)$ is a sequence of independent centered Gaussian random variables
with variance $\sigma_0^2 >0$. 
\item With probability $p_1$, only the old pole $i1$ divides. Its growth
  rate $X_{i1}$ is  linked to it's mother's one through the relation
\begin{equation}
   \label{eq:dead-1}
X_{i1}  =  \alpha_1' X_i + \beta_1' +\varepsilon_{i1}',
\end{equation}
where $\alpha_1' \in (-1,1)$, $\beta_1' \in \R$ and $(\varepsilon_{i1}',
i  \in  \T)$ is  a  sequence  of  independent centered  Gaussian  random
variables with variance $\sigma_1^2 >0$.
\item 
The sequences $((\varepsilon_{i0}, \varepsilon_{i1}), i \in \T)$,
$(\varepsilon_{i0}', i \in \T)$ and $(\varepsilon_{i1}', i \in \T)$ are
independent. 
\end{itemize}

In  Section  \ref{sec:test}, we  first  compute  the maximum  likelihood
estimator (MLE) of the parameter
\begin{equation}
   \label{eq:theta}
\theta =(\alpha_0, \beta_0, \alpha_1, \beta_1, \alpha_0', \beta_0', \alpha_1', \beta_1', p_{1,0}, p_0, p_1)
\end{equation}
and of $\kappa=(\sigma, \rho,  \sigma_0, \sigma_1)$. Then, we prove that
they  are  strongly   consistent  and  that  the  MLE   of  $\theta$  is
asymptotically  normal,   see  Proposition  \ref{prop:lan}   and  Remark
\ref{rem:cv}. Notice  that the  MLE of $(p_{1,0},  p_0, p_1)$,  which is
computed  only on the  underlying GW  tree, was  already known,  see for
example \cite{mt:imbp}.  Eventually, we  explicit a test for aging detection,
for  instance  the  null  hypothesis  $\{(\alpha_0,  \beta_0)=(\alpha_1,
\beta_1)\}$ against its alternative$\{(\alpha_0, \beta_0)\neq (\alpha_1,
\beta_1)\}$, see Proposition \ref{prop:test}. It appears that, for those
hypothesis,  using the
test statistic from \cite{g:ltbmcadca} with incomplete data due to
death cells instead of the test statistic from Proposition
\ref{prop:test} is not conservative, see Remark \ref{rem:conserve}. 

To prove those  results, we shall consider a  more general framework of
BMC which is described in  Section  \ref{par:bmc}. An important tool is
the   auxiliary   Markov   chain   which   is   defined   in   Section
\ref{par:Y}. Eventually easy to read version of our main general results
are given in Section  \ref{par:results}.

\subsection{The mathematical model of bifurcating Markov chain (BMC)}
\label{par:bmc}
We first
introduce  some  notations  related  to  the regular  binary  tree. Let
$\G_0=\{\emptyset\}$,   $\G_k=\{0,1\}^k$  for   $k\in  \N^*$,   $\T_r  =
\displaystyle \bigcup  _{0 \leq  k \leq r}  \G_k$.  The new  (resp. old)
pole daughter  of a cell $i\in \T$  is denoted by $i0$  (resp. $i1$) and
$0$  (resp. 1)  if $i=\emptyset$  is  the initial  cell or  root of  the
tree. The  set $\G_k$  corresponds to all  possible cells in  the $k$-th
generation. We  denote by  $|i|$ the generation  of $i$ ($|i|=k$  if and
only if $i\in
\G_k$). 

For  a cell  $i\in \T$,  let $X_i$  denote a  quantity of  interest (for
example its growth rate). We assume that the quantity of interest of the
daughters of  a cell $i$,  conditionally on the generations  previous to
$i$, depends only on $X_i$.  This property is stated using the formalism
of BMC.  More
precisely,  let  $(E,  \mathcal{E})$   be  a  measurable  space,  $P$  a
probability kernel  on $E \times \mathcal{E}^2$ with  values in $[0,1]$:
$P(\cdot  , A)$  is measurable  for all  $A\in \mathcal{E}^2$  and $P(x,
\cdot)$ is  a probability measure on $(E^2,\mathcal{E}^2)$,  and for any
measurable real-valued bounded function $g$ defined on $E^3$ we set
\[
Pg(x)=\int_{E^2} g(x,y,z)\; P(x,dy,dz).
\] 
\begin{defi}
  We say  a stochastic process indexed  by $\T$, $X=(X_i,  i\in \T)$, is
  a bifurcating Markov chain on a measurable space $(E, \mathcal{E})$ with
  initial distribution  $\nu$ and probability  kernel $P$, a  $P$-BMC in
  short, if:
\begin{itemize}
   \item $X_\emptyset$ is distributed as $\nu$.
   \item For  any measurable  real-valued bounded functions  $(g_i, i\in
     \T)$ defined on $E^3$, we have for all $k\geq 0$,
\[
\E\Big[\prod_{i\in \G_k} g_i(X_i,X_{i0},X_{i1}) |\sigma(X_j; j\in \T_k)\Big] 
=\prod_{i\in \G_k} Pg_i(X_{i}).
\]
\end{itemize}
\end{defi}

We  consider a metric  measurable space  $(S, \cs)$  and add  a cemetery
point to $S$, $\partial$.  Let $\bar S=S\cup\{\partial\}$ and $\bar \cs$
be the  $\sigma$-field generated by  $\cs$ and $\{\partial\}$.   (In the
biological framework  of the previous Section, $S$  corresponds to the
state space of the quantities  of interest and $\partial$ is the default
value for  dead cells.)   Let $P^*$ be  a probability kernel  defined on
$\bar S\times \bar\cs^2$ such that
\begin{equation}
   \label{eq:Pcimetiere}
P^*(\partial,\{(\partial,\partial)\})=1.
\end{equation}
Notice that this condition means that $\partial$ is an absorbing
state. (In the
biological framework  of the previous Section, condition
\reff{eq:Pcimetiere} states  that no dead cell can give birth to a
living cell.)

\begin{defi}
  Let $X=(X_i,  i\in \T)$  be a $P^*$-BMC  on $(\bar S,\bar  \cs)$, with
  $P^*$  satisfying \reff{eq:Pcimetiere}.  We call  $(X_i,  i\in \T^*)$,
  with $\T^*  = \{i \in \T,  X_i \neq \partial\}$,  a bifurcating Markov
  chain  on  a  Galton-Watson  tree.  The $P^*$-BMC  is  said  spatially
  homogeneous   if    $p_{1,0}=P^*(x,S\times   S)$,   $p_0=P^*(x,S\times
  \{\partial\})$ and  $p_1=P^*(x,\{\partial \} \times S)$  do not depend
  on $x \in S$. A spatially homogeneous $P^*$-BMC  is  said
  super-critical if $m>1$ where $m=2p_{1,0} +p_1+p_0$. 
\end{defi}

Notice that condition \reff{eq:Pcimetiere} and the spatially homogeneity
property implies that $\T^*$ is a  GW tree. This justify the name of BMC
on a Galton-Watson tree.  The GW tree is super-critical if and only if
$m>1$. From  now on, we shall only consider super-critical spatially
homogeneous  $P^*$-BMC  on  a  Galton-Watson tree.  (In  the  biological
framework  of the  previous  Section, $\T^*$  denotes  the sub-tree  of
living cells and  the notations $p_{1,0}, p_0$ and  $p_1$ are consistent
since, for instance, $P^*(x,S\times  S)$ represents the probability that
a living cell with growing rate $x$ gives birth to two living cells.)

We  now consider  the Galton-Watson  sub-tree $\T^*$.  For any  subset $J
\subset \T$, let 
\begin{equation}
   \label{eq:defJ*}
J^*= J \cap \T^*=\{j\in J; X_j\neq \partial\}
\end{equation} be the subset of $J$ of living cells
and $|J|$ be the cardinal of $J$. The process $Z= (Z_k, k\in \N)$, where
$Z_k=|\G_k^*|$, is a GW process with reproduction generating function
\[
\psi(z)= (1-p_0-p_1-p_{1,0}) +(p_{0}+ p_{1})z + p_{1,0} z^2.
\]
Notice the average  number of daughters  alive is $m$. 
We have, for $k\geq 0$,
\begin{equation}
   \label{eq:EZET}
\E[|\G_k^*|]=m^k\quad\text{and}\quad  \E[|\T_r^*|]= \sum_{q=0}^r
\E[|\G_q^*|] =\sum_{q=0}^r m^q= \frac{m^{r+1}-1}{m-1}. 
\end{equation}
Let us recall some well-known facts on super-critical GW, see
e.g. \cite{h:tbp} or \cite{an:bp}.
The  extinction probability  of  the GW  process  $Z$ is  $\displaystyle
\eta=\P(|\T^*|<\infty )=1 -\frac{m-1}{p_ {1,0}}$.  There exists a random
variable $W$ s.t. 
\begin{equation}
   \label{eq:defW}
W=\lim_{q\rightarrow\infty } m^{-q}  |\G_q^*|\quad \text{a.s. and in $L^2$},
\end{equation}
$\P(W=0)=\eta$ and whose Laplace
transform,     $\varphi(\lambda)=\E[\expp{-\lambda    W}]$,    satisfies
$\varphi(\lambda)=\psi(\varphi(\lambda/m))$ for $\lambda\geq 0$.
Notice the distribution of $W$ is completely characterized by this
functional equation and  $\E[W]=1$. 

For  $i  \in  \T$,  we  set  $\Delta_i =  (X_i,  X_{i0},  X_{i1})$,  the
mother-daughters quantities of interest.   For a finite subset $J\subset
\T$, we set
\begin{equation}
   \label{eq:defMJ}
M_J(f) = \begin{cases}
\sum_{i\in J} f(X_i) &\text{ for $f\in \cb(\bar S)$,}\\
\sum_{i\in J} f(\Delta_i) &\text{ for $f\in \cb(\bar S^3)$,}
\end{cases}
\end{equation}
with the convention that $M_\emptyset(f)=0$,  and  the following  two averages of $f$ over $J$
\begin{equation}
   \label{eq:averageMJ}
\bar M_J(f)=\inv{|J|}
M_J(f)\quad\text{if}\quad |J|>0\quad\text{and}\quad \tilde M_J(f)=\inv{\E[|J|]}
M_J(f)\quad\text{if}\quad \E[|J|]>0.
\end{equation}
We shall study the asymptotic limit of the averages of a function $f$ for
the BMC  over the $n$-th generation,  $\displaystyle \bar M_{\G_n^*}(f)$
and  $\tilde M_{\G_n^*}(f)$,  or  over  all the  generations  up to  the
$n$-th,  $\displaystyle \bar  M_{\T^*_n}(f)$  and $\displaystyle  \tilde
M_{\T_n^*}(f)$,  as $n$  goes  to  infinity. Notice  the  no death  case
studied in \cite{g:ltbmcadca} corresponds to $p_{1,0}=1$ that is $m=2$.

\subsection{The auxiliary Markov chain}
\label{par:Y}
We    define   the   sub-probability    kernel   on    $S\times   \cs^2$
$P(\cdot,\cdot)=P^*(\cdot,   \cdot  \bigcap   S^2)$  and   two
sub-probability kernels on $S\times \cs$:
\[
P^*_0=P^*(\cdot,(\cdot\bigcap  S)   \times  \bar  S)\quad\text{and}\quad
P_1^*=P(\cdot,\bar S \times (\cdot\bigcap S)).
\] 
Notice  that $P^*_0$  (resp. $P^*_1$)  is the  restriction of  the first
(resp. second)  marginal of $P^*$  to $S$. From spatial  homogeneity, we
have for  all $x\in  S$, $P(x, S^2)=p_{1,0}$ and,  for $\delta\in
\{0,1\}$,
\[
P^*_\delta(x, \{\partial\})=0\quad\text{and}\quad P^*_\delta(x,
S)=p_\delta+ p_{1,0} .
\]
We introduce an auxiliary Markov chain (see Guyon \cite{g:ltbmcadca} for
the case $m=2$). 
Let $Y=(Y_n,  n\in \N) $  be a Markov  chain on $S$ with $Y_0$
distributed  as  $X_\emptyset$ and 
transition  kernel 
\[
Q=\inv{m} (P^*_0+P^*_1).
\]
The  distribution of  $Y_n$
corresponds intuitively  to the  distribution of $X_I$  conditionally on
$I\in  \T^*$,  where  $I$ is  chosen  at  random  in $\G_n$,  see  Lemma
\ref{lem:fYn}  for a  precise  statement.  We  shall  write $\E_x$  when
$X_\emptyset=x$ (i.e. $\nu$ is the Dirac mass at $x\in S$).

Last, we need some more notations  : if $(E,\ce)$ is a metric measurable
space, then  $\cb_b(E)$ (resp.  $\cb_+(E)$)  denotes the set  of bounded
(resp.  non-negative) real-valued measurable  functions defined  on $E$.
The  set  $\cc_b(E)$ (resp.   $\cc_+(E)$)  denotes  the  set of  bounded
(resp.  non-negative)   real-valued  continuous  functions   defined  on
$E$. For a finite measure  $\lambda$ on $(E,\ce)$ and $f\in \cb_b(E)\cup
\cb_+(E)$  we shall  write $\langle  \lambda,f \rangle$  for  $\int f(x)
d\lambda(x)$.

$ $

We consider the following hypothesis $(H)$:

\noindent
\textit{The  Markov  chain  $Y$  is  ergodic, that  is  there  exists  a
  probability measure  $\mu$ on $(S,\cs)$ s.t., for  all $f\in \cc_b(S)$
  and   all   $x\in   S$,  $\displaystyle   \lim_{k\rightarrow\infty   }
  \E_x[f(Y_k)]= \langle \mu,f \rangle$.}

$ $

Notice that  under $(H)$,  the probability measure  $\mu$ is  the unique
stationary  distribution  of  $Y$  and  $(Y_n, n\in  \N)$  converges  in
distribution to $\mu$.

\subsection{The main results}
\label{par:results}
We can  now state our principal results on the  weak law of large
numbers and fluctuations  for the averages over a generation  or up to a
generation.  Those results are a particular case of the more general
statements given in Theorem \ref{th:wlln4} and Theorem \ref{theo:stable}, using Remark \ref{rem:H}.

\begin{theo}
\label{theo:res-intro}
Let  $(X_i,  i\in  \T^*)$  be  a  super-critical  spatially  homogeneous
$P^*$-BMC on a GW tree and  $W$ be defined by \reff{eq:defW}.  We assume
that $(H)$ holds  and that $x\mapsto P^*g(x) \in  \cc_b(\bar S)$ for all
$g \in \cc_b(\bar S^3)$. Let $f\in \cc_b(\bar S ^3)$.
\begin{itemize}
   \item \textbf{Weak law of large numbers.} We have the following convergence
     in probability:
\begin{align*}
 \ind_{\{|\G_r^*|>0 \}} \inv{|\G_r^*|}
\sum_{i\in \G^*_r}f(\Delta_i) 
   &\;\xrightarrow[r\rightarrow \infty ]{\P}\;\langle \mu, P^* f \rangle
   \ind_{\{W\neq 0 \}} ,  \\
\ind_{\{|\G_r^*|>0 \}} \inv{|\T_r ^*|}
\sum_{i\in \T^*_r}f(\Delta_i)
&   \;\xrightarrow[r\rightarrow \infty ]{\P}\;\langle \mu, P^* f \rangle
   \ind_{\{W\neq 0 \}} .  
\end{align*}
   \item \textbf{Fluctuations.} We have the following convergence in
     distribution:  
\[
\ind_{\{|\G^*_r|>0\}} \inv{\sqrt{|\T^*_r|}} \sum_{i\in \T^*_r}
\Big(f(\Delta_i) - P^* f(X_i) \Big)
   \;\xrightarrow[r\rightarrow \infty ]{\text{(d)}}\;
\ind_{\{W\neq 0\}} \sigma G,
\]
where $\sigma^2=\langle \mu, P^* (f^2) - (P^* f)^2\rangle $ and $G$ is a
Gaussian random variable with  mean zero, variance~$1$, and independent of
$W$.
\end{itemize} 
\end{theo}

One can  get the strong law  of large numbers under  stronger hypothesis on
$Y$  (such  as  geometric  ergodicity)  using similar  arguments  as  in
\cite{g:ltbmcadca}.  We  also can prove that the  fluctuations over each
generation are asymptotically independent.

\begin{theo}
Let  $(X_i,  i\in  \T^*)$  be  a  super-critical  spatially  homogeneous
$P^*$-BMC on a  GW tree.   We assume that  $(H)$ holds and that $x\mapsto  P^*g(x) \in \cc_b(\bar
  S)$ for all $g \in \cc_b(\bar  S^3)$.  Let $d\geq 1$, and for $\ell\in
  \{1,    \ldots,    d\}$,   $f_\ell    \in    \cc_b(\bar   S^3)$    and
  $\sigma_\ell^2=\langle \mu, P^*  (f_\ell^2) - (P^* f_\ell)^2\rangle $.
  We set for $f\in \cc_b(\bar S^3)$
\[
N_n(f)=\inv{\sqrt{|\G^*_n|}} \sum_{i\in \G^*_n} \Big( f(\Delta_i) - P^*
f(X_i)\Big).
\]
Then we have the following convergence in distribution: 
\[
\left(N_n(f_1), \ldots, N_{n-d+1}(f_d)\right)\ind_{\{|\G^*_n|>0\}}
   \;\xrightarrow[n\rightarrow \infty ]{\text{(d)}}\;
\ind_{\{W\neq 0\}} (\sigma_1 G_1, \ldots, \sigma_d G_d),
\]
where $G_1, \ldots,  G_d$ are independent Gaussian random variables with
mean zero and variance~$1$ and are independent of $W$ given by \reff{eq:defW}.
\end{theo}

Even if the results  on fluctuations in Theorem \ref{theo:res-intro} are
not  complete, see  the  Remark \ref{rem:intro}  below,  they are  still
sufficient to  study the statistical model  we gave in Section 
\ref{ssec:stat} for the detection of  cellular aging  from  cell lineage
when  death of  cells can occur.

\begin{rem}
\label{rem:intro}
Let $V=(V_r, r\geq 0)$ be a  Markov chain on a finite state space. We
assume  $V$ is irreducible, with transition
matrix  $R$ and unique invariant probability $\mu$. Then it is well
known, see \cite{mt:mcss},
 that $\inv{r}\sum_{i=1}^r h(V_i)$ converges a.s. to $\langle \mu,h
\rangle$ and that, to prove the fluctuations result, one solves the
Poisson equation $H-RH=h-\langle \mu,h \rangle$, writes
\begin{equation}
   \label{eq:meyn}
\inv{\sqrt{r} }\sum_{i=1}^r \Big(h(V_i) - \langle \mu,h
\rangle\Big)
=  \inv{\sqrt{r}
}\sum_{i=1}^r \Big(H(V_i) - RH(V_{i-1})\Big)  +\inv{\sqrt{r}} RH(V_0) -
\inv{\sqrt{r}}RH(V_r),   
\end{equation}
and then uses martingale theory  (we use similar techniques to prove the
fluctuations in Theorem  \ref{theo:res-intro}) to obtain the asymptotic
normality  of $  \inv{\sqrt{r} }\sum_{i=1}^r H(V_i)  - RH(V_{i-1})
$. It  then only  remains to  say that $\inv{\sqrt{r}}  RH(V_0) $  and $
\inv{\sqrt{r}}RH(V_r)$ converge to $0$ to conclude.

Assume  that hypothesis  of  Theorem \ref{theo:res-intro}  hold and  that
$x\mapsto  P^*(x,A)$ is continuous  for all  $A\in \cb(\bar  S^2)$.  Let
$h\in  \cc_b(\bar   S)$.   Theorem  \ref{theo:res-intro}   implies  that
$\ind_{\{|\G_r^*|>0  \}}  \inv{|\T_r  ^*|}  \sum_{i\in  \T^*_r}h(X_i)  $
converges  in probability  to $\langle  \mu, h  \rangle  \ind_{\{W\neq 0
  \}}$.  To  get the fluctuations,  that is the limit  of 
\[
\ind_{\{|\G_r^*|>0  \}} \inv{\sqrt{|\T_r  ^*|}}  \sum_{i\in \T^*_r}\Big(
h(X_i) -\langle  \mu, h \rangle\Big)
\]
as $r$ goes to infinity, using martingale  theory, one can
think of using the  same kind of approach in order to  use the result on
fluctuations of Theorem \ref{theo:res-intro}.  But then notice that what
will correspond  to the boundary term  in \reff{eq:meyn} at  time $r$, $
\inv{\sqrt{r}}RH(V_r)$,  will  now be  a  boundary  term  over the  last
generation   $\G^*_r$  whose   cardinal  is   of  the   same   order  as
$|\T^*_r|$. Thus the order of  the boundary term is not negligible,
which unable us to conclude.  

The
fluctuations for $\sum_{i\in \T^*_r}h(X_i) $ are still an open question.
\end{rem}

\subsection{Organization of the paper}
We quickly study  the auxiliary chain in Section  \ref{sec:Yn}. We state
the  first  result   on  the  weak  law  of   large  number  in  Section
\ref{sec:wlln}.    Section  \ref{sec:technical}   is  devoted   to  some
preparatory  results  in order  to  apply  results  on fluctuations  for
martingale. Our  main result,  Theorem \ref{theo:stable}, is  stated and
proved  in  Section  \ref{sec:tcl}.   The biological  model  of  Section
\ref{ssec:stat} is analyzed in Section \ref{sec:test}.

\section{Preliminary result and notations}
\label{sec:Yn}
Recall the Markov chain $Y$ defined in Section \ref{par:Y}. 
\begin{lem}
\label{lem:fYn}
We have, for $f\in \cb_b(S)\cup \cb_+(S)$,
\begin{equation}
   \label{eq:fYn}
 \E[f(Y_n)]= m^{-n} \sum_{i\in \G_n}
 \E[f(X_i)\ind_{\{i\in 
     \T^*\}}]= \frac{\sum_{i\in \G_n} \E[f(X_i)\ind_{\{i\in
     \T^*\}}]}{\sum_{i\in \G_n} \P(i\in \T^*)} =\E[f(X_I)|I\in \T^*],
\end{equation}
where  $I$  is  a uniform  random variable on  $\G_n$ independent of  $X$.
\end{lem}

\begin{proof}
   
We consider the first equality. 
Recall that $Y_0$ has distribution $\nu$.   For
$i=i_1 \ldots i_n \in \G_n$, we have, thanks to
\reff{eq:Pcimetiere} and the definition of $P^*$, 
\[
\E[f(X_i)\ind_{\{i\in 
     \T^*\}}]=\E[f(X_i)\ind_{\{X_i\not= \partial\}}]=
  \langle \nu, \left(P^*_{i_1}
  \ldots P^*_{i_n}\right) f \rangle,
\]
so that 
\begin{align*}
\sum_{i\in \G_n}
 \E[f(X_i)\ind_{\{i\in 
     \T^*\}}]
&= \sum_{i_1, \ldots, i_n \in \{0,1\}}  \langle \nu, \left(P^*_{i_1} 
  \ldots P^*_{i_n}\right) f \rangle\\
&= \langle \nu, \left(P_0^* +  P_1^*\right)^n f \rangle=  m^n \langle \nu, Q^n
f \rangle=m^n \E[f(Y_n)]. 
\end{align*}
This gives the first equality.  Then take $f=1$ in the previous equality
to get $m^n=\sum_{i\in  \G_n} \P(i\in \T^*)$ and the  second equality of
\reff{eq:fYn}. The last equality of \reff{eq:fYn} is obvious.

\end{proof}

We recall that $\nu$ denotes the distribution of $X_\emptyset$. 
Any function $f$ defined on $S$ is extended to $\bar S$ by setting
$f(\partial)=0$. 
Let $F$ be a vector subspace of $\cb(S)$ s.t. 
\begin{itemize}
   \item[$(i)$] $F$ contains the constants;
   \item[$(ii)$] $F^2:=\{f^2; f\in F\} \subset F$;
   \item[$(iii)$]
    \begin{itemize}
      \item $F\otimes F\subset L^1(P(x, \cdot))$ for all $x\in S$
     and $P(f_0\otimes f_1) \in  F$ for all $f_0,f_1\in F$;
   \item For  $\delta\in \{0,1\}$, $F\subset  L^1(P^*_\delta(x, \cdot))$
     for all $x\in S$ and $P^*_\delta (f) \in F$ for all $f\in F$;
    \end{itemize} 
   \item[$(iv)$] There exists a probability measure $\mu$ on $(S, \cs)$
     s.t. $F \subset L^1(\mu)$ and $\displaystyle  \lim_{n\rightarrow \infty}
     \E_x[f(Y_n)]=\langle \mu,f \rangle$ for all $x\in S$ and $f \in F$;
   \item[$(v)$] For all $f\in  F$, there exists $g\in F$ s.t. for all
     $r\in \N$, $|Q^r f | \leq  g$;
   \item[$(vi)$] $F\subset L^1(\nu)$. 
\end{itemize}
By convention a function defined on $\bar S$ is said to belong to $F$ if
its restriction to $S$ belongs to $F$. 
\begin{rem}
   \label{rem:H}
   Notice  that if  $(H)$  is  satisfied and  if  $x\mapsto P^*g(x)$  is
   continuous  on $S$  for  all $g  \in  \cc_b(\bar S^3)$  then the  set
   $\cc_b(S)$ fulfills $(i)-(vi)$.
\end{rem}

\section{Weak law of large numbers}
\label{sec:wlln}
We give the first result of this section. Recall notations
\reff{eq:defJ*} and \reff{eq:defMJ}. 

\begin{theorem}   \label{th:wlln1}   Let  $(X_i,   i\in   \T^*)$  be   a
  super-critical spatially  homogeneous $P^*$-BMC on a GW  tree. Let $F$
  satisfy (i)-(vi)  and $f\in F$.  The  sequence $(\tilde M_{\G_q^*}(f),
  q\in  \N)$ converges to  $\langle \mu,  f \rangle  W$ in  $L^2$, where
  $W$ is defined by  \reff{eq:defW}.  We also have that  the sequence $(
  \bar   M_{\G_q^*}(f)\ind_{\{|\G_q^*|>0\}},  q\in  \N)$   converges  to
  $\langle \mu, f \rangle \ind_{\{W\neq 0\}}$ in probability.
\end{theorem}

\begin{proof}
We first assume that  $\langle \mu ,f \rangle=0$. We have, 
\[
\|\sum_{i \in \G_q^*} f(X_i)\|_{L^2}^2 
 =  \E\Big[(\sum_{i \in \G_q}  f(X_i)\ind_{\{i \in \T^*\}})^
   2\Big] 
 =  \sum_{i \in \G_q} \E[f^2(X_i)\ind_{\{i \in \T^*\}}] +
B_q  
=m^q \E[f^2(Y_q)]+B_q,
\]
with     $\displaystyle    B_q=\sum_{(i,j)\in    \G_q^2,     i\neq    j}
\E[f(X_i)f(X_j)\ind_{\{(i,j)   \in    {\T^*}^2\}}]$,   where   we   used
\reff{eq:fYn} for the last equality.

Since the sum in $B_q$ concerns all pairs of
distinct elements of $\G_q$, we have that $i \wedge j$, the most recent
common ancestor of $i$ and $j$, does not belong to $\G_q$. We shall compute
$B_q$ by decomposing according to the generation of $k=i \wedge j$:
$ \displaystyle B_q  
=  \sum_{r=0}^{q-1} \sum_{k \in \G_r} C_k$ with 
\[
C_k
=  \sum_{(i,j)\in \G_q^2,
  i\wedge j = k} \E[f(X_i)f(X_j)\ind_{\{(i,j) \in {\T^*}^2\}} ] .
\]

If  $|k|= q-1$, using  the Markov property of $X$ and of the
GW process at generation $q-1$, we get
\begin{align*}
C_k
&=\sum_{(i,j)\in \G_{1}^2,
  i\wedge j = \emptyset } \E[\E_{X_k} [f(X_i)f(X_j)\ind_{\{(i,j) \in
  {\T^*}^2\}} ]  \ind_{\{k\in \T^*\}}]\\
&=2\E[P(f\otimes 
f)(X_k)\ind_{\{k\in \T^*\}}] .  
\end{align*}

If  $|k|< q-1$, we have, with $r=|k|$, 
\begin{align*}
C_k
&=2\sum_{(i,j)\in \G_{q-r-1}^2}  \E[\E_{X_{k0}} [f(X_i)\ind_{\{i\in
  \T^*\}}] \E_{X_{k1}} 
  [f(X_j)\ind_{\{j\in \T^*\}}]   \ind_{\{k0\in \T^*, k1\in \T^*\}}]\\
&=2 \E[\sum_{i\in \G_{q-r-1}} \E_{X_{k0}} [f(X_i)\ind_{\{i\in \T^*\}}]
\sum_{j\in \G_{q-r-1}} \E_{X_{k1}} 
  [f(X_j)\ind_{\{j\in \T^*\}}]   \ind_{\{k0\in \T^*, k1\in \T^*\}}]\\
&=2 m^{2(q-r-1)}  \E[\E_{X_{k0}} [f(Y_{q-r-1})]  \E_{X_{k1}}
[f(Y_{q-r-1})]     \ind_{\{k0\in \T^*, k1\in \T^*\}}]\\ 
&= 2  m^{2(q-r-1)} \E[P(Q^{q-r-1} f \otimes Q^{q-r-1} f
)(X_k)\ind_{\{k\in \T^*\}} ],
\end{align*}
where we used the Markov   property of $X$ and of the
GW process at generation $r+1$ for the first equality,  \reff{eq:fYn}
for the third equality and the Markov property at generation $r$ for the
last equality. 

In particular, we get that $\displaystyle  C_k=2m^{2(q-r-1)} \E[P(Q^{q-r-1} f \otimes Q^{q-r-1} f
)(X_k)\ind_{\{k\in \T^*\}} ]$ for all $k$ s.t. $|k|\leq  q-1$. 
Using 
\reff{eq:fYn}, 
we deduce that 
\begin{align*}
  B_q  &=  2    \sum_{r=0}^{q-1}   m^{2(q-r-1)} \sum_{k\in  \G_r}
  \E[P(Q^{q-r-1} f \otimes Q^{q-r-1} f
  )(X_k)\ind_{\{k\in \T^*\}} ]\\
  &= 2   \sum_{r=0}^{q-1}   m^{2q -r -2} \langle \nu, Q^r P(Q^{q-r-1} f
  \otimes Q^{q-r-1} f )\rangle .
\end{align*}

Therefore, we get 
\begin{align}
\nonumber
\|\tilde M_{\G_q^*}(f)\| _{L^2}^2 
&= m^{-2q} 
\|\sum_{i \in \G_q^*} f(X_i)\|_{L^2}^2 \\
&= m^{-q} \E[f^2(Y_q)]+ 2  m^{-2}  \sum_{r=0}^{q-1}   m^{-r} \langle \nu,  
Q^r P(Q^{q-r-1} f  
  \otimes Q^{q-r-1} f )  \rangle.
\label{eq:ml2mg}
\end{align}
As   $f\in  F$,  properties   $(ii)$,  $(iv)$, $(v)$   and  $(vi)$   imply  that
$\displaystyle  \lim_{q\rightarrow   \infty  }  m^{-q}  \E[f^2(Y_q)]=0$.
Properties  $(iii)$, $(iv)$ and  $(v)$ with  $\langle \mu,f  \rangle =0$
implies that $ P(Q^{q-r-1} f \otimes  Q^{q-r-1} f) $ converges to $0$ as
$q$ goes to infinity (with $r$  fixed) and is bounded uniformly in $q>r$
by  a function  of $F$.  Thus, properties  $(v)$ and  $(vi)$  imply that
$\langle  \nu,  Q^r  P(Q^{q-r-1}  f  \otimes Q^{q-r-1}  f  )  \rangle  $
converges to $0$ as $q$ goes to infinity (with $r$ fixed) and is bounded
uniformly  in   $q>r$  by   a  finite  constant,   say  $K$.    For  any
$\varepsilon>0$, we  can choose $r_0$ s.t.  $\sum_{r>  r_0} m^{-r} K\leq
\varepsilon$ and $q_0>r_0$ s.t. for $q\geq q_0$ and $r\leq r_0$, we have
$\displaystyle  |\langle \nu, Q^r  P(Q^{q-r-1} f  \otimes Q^{q-r-1}  f )
\rangle|\leq \varepsilon/r_0$.  We then get that for all $q \ge q_0$
\[
\sum_{r=0}^{q-1}  m^{-r}  |\langle  \nu  ,  Q^r  P(Q^{q-r-1}  f  \otimes
Q^{q-r-1}  f  )   \rangle|  \leq  \sum_{r=0}^{r_0}  r_0^{-1}\varepsilon+
\sum_{r=r_0+1}^{q-1} m^{-r} K \leq 2\varepsilon.
\]
This   gives    that   $\displaystyle   \lim_{q\rightarrow    \infty   }
\sum_{r=0}^{q-1}  m^{-r} \langle \nu ,  Q^r P(Q^{q-r-1}  f \otimes  Q^{q-r-1} f
) \rangle=0$.   Eventually, we  get from  \reff{eq:ml2mg} that  if
$\langle \mu,f \rangle =0$, 
then $\displaystyle \lim_{q\rightarrow \infty } \|\tilde M_{\G_q^*}(f)\|
_{L^2}=0$.

For any function $f \in F$, we have, with $g= f -\langle \mu , f \rangle$, 
\[
\tilde  M_{\G_q^*}(f)= \tilde  M_{\G_q^*}(g) + \langle \mu,f \rangle  m^{-q} |\G^*_q|.
\]
As $g\in F$  and $\langle \mu , g  \rangle=0$, the previous computations
yield   that  $\displaystyle   \lim_{q\rightarrow   \infty  }   \|\tilde
M_{\G_q^*}(g)\| _{L^2}=0$.  As $(m^{-q} |\G_q^*|, q \ge 1)$ converges in
$L^2$ (and a.s.)   to $W$, we get that  $\tilde M_{\G_q^*}(f)$ converges
to $\langle \mu, f \rangle W$ in $L^2$.

Then use that $m^{-q} |\G^*_q|$ converges a.s. to $W$ to get the second
part of the Theorem. 
\end{proof}

We now prove a similar result for the average over the $r$-th first
generations. We set $t^r= \E[|\T_r^*|]$, see \reff{eq:EZET}. We first
state en elementary Lemma whose proof is left to the reader.  
\begin{lem}  
\label{l2}  
Let  $(v_r, r   \in  \N)$  be  a  sequence  of
  non negative real  numbers converging to $a  \in \R_+$, and  $m$ a real
  such that $m > 1$. Let
$$w_r = \sum_{q=0}^r m^{q-r-1}v_q.$$
Then the sequence $(w_r, r \in \N)$ converges to $a/(m-1)$.
\end{lem}

Recall notations
\reff{eq:defJ*} and \reff{eq:defMJ}. 
\begin{theorem}   \label{th:wlln2}   Let  $(X_i,   i\in   \T^*)$  be   a
  super-critical spatially homogeneous $P^*$-BMC  on a GW tree.  Let $F$
  satisfy (i)-(vi)  and $f\in F$.  The  sequence $(\tilde M_{\T_r^*}(f),
  r\in \N)$ converges to $\langle \mu,  f \rangle W$ in $L^2$, where $W$
  is defined by \reff{eq:defW}.  We  also have that the sequence $( \bar
  M_{\T_r^*}(f)\ind_{\{|\T_r^*|>0\}}, r  \in \N)$ converges  to $\langle
  \mu, f \rangle \ind_{\{W\neq 0\}}$ in probability.
\end{theorem}

\begin{proof}
We have 
\begin{align*}
\left \| \frac{1}{t_r} \sum_{i \in \T_r^*}f(X_i) - \langle \mu, f \rangle W \right
\|_{L^2} 
&= \left \| \sum_{q=0}^r \frac{m^q}{t_r} \Big(\tilde  M_{\G^*_q} (f) -
  \langle \mu, f \rangle W \Big)
\right\|_{L^2}\\  
&\le  \sum_{q=0}^r \frac{m^q}{t_r}  \left \| \tilde M_{\G^*_q} (f) -
  \langle \mu, f \rangle W 
\right\|_{L^2}\\
& =  \frac{m-1}{1-m^{-r-1}} 
\sum_{q=0}^r m^{q-r-1}  \left \| \tilde  M_{\G^*_q} (f) -
  \langle \mu, f \rangle W 
\right\|_{L^2}.
\end{align*}
The first part of the Theorem  follows from Theorem \ref{th:wlln1} and
Lemma \ref{l2}.

Use that $m^{-q}  |\G^*_q|$ converges a.s. to $W$  to deduce  that $t_r^{-1}
|\T^*_r|$ converges a.s. to $W$, and thus  get the second part of the Theorem.
\end{proof}

We  end this  section, with  an extension  of the  results  to functions
defined on the mother-daughters  quantities of interest $\Delta_i= (X_i,
X_{i0},  X_{i1})\in  \bar S^3$.   Recall  notations \reff{eq:defJ*}  and
\reff{eq:defMJ}.

\begin{theorem} 
\label{th:wlln4} 
Let  $(X_i,  i\in  \T^*)$  be  a  super-critical  spatially  homogeneous
$P^*$-BMC on  a GW tree. Let  $F$ satisfy (i)-(vi) and  $f\in \cb(\bar S
^3)$.  We  assume that  $P^*f$ and $P^*(f^2)$  exist and belong  to $F$.
Then  the  sequences  $(\tilde  M_{\G_q^*}(f), q\in  \N)$  and  $(\tilde
M_{\T_r^*}(f), r \in  \N)$ converge to $\langle \mu,  P^*f \rangle W$ in
$L^2$, where $W$ is defined  by \reff{eq:defW}; and the sequences $(\bar
M_{\G_q^*}(f)\ind_{\{|\G_q^*|>0    \}},    q\in    \N)$    and    $(\bar
M_{\T_r^*}(f)\ind_{\{|\G_r^*|>0  \}}, r  \in \N)$  converge  to $\langle
\mu, P^*f \rangle \ind_{\{W \neq 0\}}$ in probability.
\end{theorem}

\begin{proof}
  Recall that $M_{\G_q^*}(f) =  \sum_{i \in \G_q^*} f(\Delta_i)$. The
  Markov property for BMC gives
\[\|M_{\G_q^*}(f) \|_{L^2}^2 = \|M_{\G_q^*}(P^*f) \|_{L^2}^2 +
\E[M_{\G_q^*}(P^*(f^2)-(P^*f)^2)].
\]
Since  $(m^{-q} M_{\G_q^*}(P^*(f^2)-(P^*f)^2), q  \in \N)$  converges to
$\langle \mu, P^*(f^2)-(P^*f)^2 \rangle$ in  $L^2$ and thus in $L^1$, we
have that $m^{-2q}\E[M_{\G_q^*}(P^*(f^2)-(P^*f)^2)]$ converges to $0$ as
$q$  goes to  infinity.  Then,  we deduce  the  convergence of  $(\tilde
M_{\G_q^*}(f), q\in \N)$ and $(\bar M_{\G_q^*}(f)\ind_{\{|\G^*_r|>0 \}},
q\in \N)$ from Theorem \ref{th:wlln1}.

The proof for the convergence of $(\tilde M_{\T_r^*}(f), r \in \N)$ and  $(\bar
M_{\T_r^*}(f)\ind_{\{|\G^*_r|>0 \}}, r \in  \N)$ mimics then the proof of
Theorem \ref{th:wlln2}.
\end{proof}

\section{Technical results about the weak law of large numbers}
\label{sec:technical}

The technical  Propositions of this section  deal with the  average of a
function   $f$   when  going   through   $\T^*$  \emph{via}   timescales
$(\tau_n(t), t  \in [0,1])$  preserving the genealogical  order. Roughly
speaking, these timescales allow to  visit the sub-tree $\T^*$. In order
to define $(\tau_n(t),  t \in [0,1])$ we need to  define $I_n^*$, set of
the  $n$ ``first''  cells of  $\T^*$.  
Let  $(X_i,  i\in  \T^*)$  be  a  super-critical  spatially  homogeneous
$P^*$-BMC on  a GW tree and   $\cg$ be  the $\sigma$-field
generated by $(X_i, i \in \T)$.
\begin{itemize}
\item We consider random variables
  $(\Pi_q^*,  q\in \N^*)$ which are conditionally on $\cg$,  independent
  and s.t.   $\Pi_q^*$ is  distributed as a uniform random  permutation on
  $\G_q^*$.  In  particular, given  $|\G_q^*|=k$,
  $(\Pi_q^*(1), \ldots,  \Pi_q^*(k))$ can be viewed as  a random drawing
  of all the elements of $\G_q^*$, without replacement.

\item  For each integer  $n \in \N^*$, we  define the  random variable
  $\displaystyle \rho_n=\inf\{ k; n\leq |\T^*_k|\}$, with the convention
  $\inf \emptyset =\infty $. Loosely speaking, $\rho_n$ is the number of
  the generation to which belongs  the $n$-th element of $\T^*$.  Notice
  that $\rho_1=0$.

\item Let $\tilde \Pi$ be the  application from  $\N^*$ to $\T^*
\cup \{\partial_\T\}$, where $\partial_\T $ is a cemetery point added to
$\T^*$,  given by 
$\tilde{\Pi}(1)=\emptyset$ and for $k \ge 2$ :
\[
\tilde{\Pi}(k) = 
\begin{cases}
 \Pi_{\rho_k}^*(k-|\T_{\rho_k-1}^*|) & \text{if $\rho_k < +\infty $}\\
  \partial_\T & \text{if $\rho_k = +\infty $.}
\end{cases}
\]
\end{itemize} 

Notice that $\tilde \Pi$ defines a
random order on $\T^*$ which preserves the genealogical order: if $k\leq
n$ then $|\tilde \Pi (k)|\leq |\tilde \Pi(n)|$, with the convention
$|\partial_\T|=\infty $. 
We thus define the set of the $n$ ``first'' elements of $\T^*$ (when $|\T^*|\geq n$): 
\begin{equation}
   \label{eq:nfirst}
I_n^* = \{\tilde{\Pi}(k), 1 \le k \le n \wedge |\T^*|\}.
\end{equation}

We can  now introduce  the timescales:  for $n \ge  1$, we  consider the
subdivision  of   $[0,1]$  given  by  $\{0,s_n,   \ldots,  s_0\}$,  with
$s_k=m^{-k}$. We define the continuous random time change $(\tau_n(t), t
\in [0,1])$ by 
\begin{equation}
   \label{eq:snt}
\tau_n(t)=\begin{cases}
 m^n t,& t\in [0,  m^{-n}], \\
|\T^*_{n-k}|+(m^kt-1)(m-1)^{-1}|\G^*_{n-k+1}|, &  t\in
[m^{-k}, m^{-k+1}], 1\leq
k\leq n.
\end{cases}
\end{equation}
Notice   that    $\tau_n(t)\leq   |\T^*|$.    The    set   $I^*_{\lfloor
  \tau_n(t)\rfloor }$, with $t\in [0,1]$, corresponds to the elements of
$\T^*_{n-k}$, with $k=\lfloor -\frac{ \log(t)}{\log(m)} \rfloor +1$, and
the ``first'' fraction $(m^k t-1)/(m-1)  $ of the elements of generation
$\G^*_{n-k+1}$.

Recall  \reff{eq:defMJ}. 
For the sake of  simplicity, for any real $x \ge 0$, we
will write $M_x^*(f)$ instead  of $M_{I^*_{\lfloor x \rfloor}}(f)$, with
the convention that $M_0^*(f)=0$.

\begin{prop}
\label{prop:snt}
  Let  $F$ satisfy  (i)-(vi),  $f\in F$ and $t \in  [0,1]$. The
  sequence  $(m^{-n}M^*_{ \tau_n(t)}(f)$, $n  \in \N^*)$
  converges to $\langle \mu , f \rangle m(m-1)^{-1}  W t$ in $L^2$.
\end{prop}

\begin{proof}
  We first consider  the case  $\langle \mu,f \rangle =0$. If $t=0$, then
  $\tau_n(t)=0$ and $M^*_0(f)=0$ by convention. Let $t \in (0,1]$. There
  exists a  unique $k \ge 1$ such  that $m^{-k} < t  \le m^{-k+1}$.  For
  $n\geq  k$, we  have,  using 
  \reff{eq:snt} and that $\tilde{\Pi}$ preserves the order on $\T^*$,
\[
 M^*_{\tau_n(t)}(f)= \sum_{i\in I_{\lfloor \tau_n(t) \rfloor}^*}   f(X_i) =\sum_{i=1}^{\lfloor \tau_n(t) \rfloor} f(X_{\tilde{\Pi}(i)}) =
M_{\T^*_{n-k}} (f) +M_{J_n}(f),
\]
where $J_n  =   \{  \tilde{\Pi}(i),     |\T^*_{n-k}|<i\leq   \lfloor
\tau_n(t) \rfloor  \}$. Notice that $J_n=\emptyset$ if
$|\T^*_{n-k+1}|=0$ and that, by convention, we then have $M_{J_n}(f)=0$. 
Theorem  \ref{th:wlln2}  implies that  $m^{-n} M_{\T^*_{n-k}} (f)$
converges  to $0$ in $L^2$  as $n$ goes to  $\infty$.   
  Recall $\cg$ is  the $\sigma$-field generated
by $(X_i, i \in \T)$. Since $J_n \subset \G^*_{n-k+1}$, we have
\[
\E[M_{J_n}(f)^2  | \mathcal{G}] = \sum_{i,j \in  \G^*_{n-k+1}}
f(X_i)f(X_j) \E[\ind_{\{i,j \in J_n  \}}|\T^*].
\]
Thanks to the definition of $\tilde{\Pi}$, we have for $i,j\in
\G_{n-k+1}$ 
\[
\ind_{\{i,j\in \G^*_{n-k+1}\}} \E[\ind_{\{i,j\in J_n\}}|\T^*]
=\ind_{\{i,j\in \G^*_{n-k+1}\}} (\ind_{\{i \neq
  j\}} \chi_2 + \ind_{\{i=j\}}\chi_1), 
\]
where, with $a=\lfloor
(m^kt-1)(m-1)^{-1}|\G^*_{n-k+1}| \rfloor$, 
\[
\chi_{1}= \frac{a}{|\G^*_{n-k+1}|}\quad\text{and}\quad 
\chi_2=\frac{a(a-1)}{|\G^*_{n-k+1}|(|\G^*_{n-k+1}|-1)}.
\]
Thus, we get 
\begin{align*}
   \E[M_{J_n}(f)^2  | \mathcal{G}] 
&= \chi_2  \sum_{i,j \in
  \G^*_{n-k+1}} f(X_i)f(X_j) + (\chi_1 -\chi_2) \sum_{i \in
  \G^*_{n-k+1}} f^2(X_i) \\
&=\chi_2 M_{\G^*_{n-k+1}}( f)^2 + (\chi_1
-\chi_2)M_{\G^*_{n-k+1}}(f^2)\\
&\leq  M_{\G^*_{n-k+1}}( f)^2 + M_{\G^*_{n-k+1}}(f^2),
\end{align*}
as $0\leq \chi_2\leq \chi_1\leq 1$. We have 
\begin{equation}
   \label{eq:majoL2snt}
\|m^{-n}M_{J_n}(f)^2  \|_{L^2}^2  \le
\|m^{-n}M_{\G^*_{n-k+1}}(f)\|_{L^2}^2 + m^{-n}
\|m^{-n}M_{\G^*_{n-k+1}}(f^2) \|_{L^1}.
\end{equation}
The first  term of the right hand-side  of \reff{eq:majoL2snt} converges
to $\langle \mu,f \rangle W=0$ as $n$ goes to infinity, thanks to Theorem
\ref{th:wlln1}. The same Theorem entails that
$\|m^{-n}M_{\G^*_{n-k+1}}(f^2) \|_{L^1}$ converges  to $\E[\langle \mu ,
f^2 \rangle W]$, and consequently the second term of the right hand-side
of \reff{eq:majoL2snt} also converges to $0$ as $n$ goes to infinity. 
We deduce that  the sequence $(m^{-n}M_{J_n}(f)^2, m\in \N^*)$
converges  to $0$ in $L^2$. 

Since   $m^{-n}   M^*_{\tau_n(t)}(f)=    m^{-n}    M_{\T^*_{n-k}}(f)   +
m^{-n}M_{J_n}(f) $,   the  sequence $(m^{-n}M^*_{\tau_n(t)
}(f), n \in \N^*)$ converges to $0$ in $L^2$.

Next, we consider the case $\langle \mu, f \rangle \neq 0$. We set $g=f-
\langle \mu, f \rangle$. Since $ m^{-n}M^*_{\tau_n(t)}(f)= m^{-n}M^*_{
  \tau_n(t) }(g)  + \langle \mu, f \rangle  m^{-n} \lfloor \tau_n(t)
\rfloor $, the Proposition  will be proved  as soon as we  check that $(m^{-n}
\lfloor \tau_n(t) \rfloor, n \in \N^*)$ converges to $m(m-1)^{-1} t W$
in $L^2$.

The case $t=0$ is obvious. For $t \in (0,1]$, there  exists a unique $k
\ge 1$  such that $m^{-k} < t \le m^{-k+1}$. 
We deduce from \reff{eq:snt}  that, for $1\leq k\leq n$,
\[
m^{-n}\tau_n(t) = (m-1)^{-1} 
\left(\frac{|\T^*_{n-k}|}{t_{n-k}}(m^{-k+1}-\frac{1}{m^n}) +
  \frac{|\G^*_{n-k+1}|}{m^{n-k+1}}(mt-m^{-k+1}) \right).
\]
Since both  $m^{-n}|\G^*_n|$ and $t_n^{-1}|\T^*_n|$ converges  to $W$ in
$L^2$,   we  finally  obtain   that  $m^{-n}\tau_n(t)$   converges  to
$m(m-1)^{-1}tW$ in $L^2$.
\end{proof}

We deduce from \reff{eq:snt} and \reff{eq:EZET}, that for $t\in (0,1]$,
$n\geq k$, where 
$k= \lfloor -\frac{\log(t)}{\log(m)}\rfloor +1$, we have 
\[
\E[\tau_n(t)]=t_{n-k} + (m^k t-1)(m-1)^{-1} m^{n-k+1}=(m^{n+1} t-1)(m-1)^{-1}.
\]
Thus,  Proposition  \ref{prop:snt}  implies  that  $(\E[\tau_n(t)]^{-1}
M^*_{  \tau_n(t) }(f),  n \in  \N^*)$ converges  to $\langle  \mu  , f
\rangle W$ in $L^2$ for all $t\in [0,1]$, which generalizes Theorem \ref{th:wlln2}.

In fact the convergence in Proposition \ref{prop:snt} is uniform in
$t$. 

\begin{cor}
\label{cor:wllunif}
  Let  $F$ satisfy (i)-(vi),  $f\in F$  s.t. $|f|\in  F$. We  set $R_n(t)=
  m^{-n}M^*_{ \tau_n(t)}(f)- \langle \mu , f \rangle m(m-1)^{-1}  W t$.
  The sequence  $(\sup_{t\in [0,1]} | R_n(t)|,n  \in \N^*)$ converges  to 0 in
  $L^2$.
\end{cor}

\begin{proof}
  Let   $f\in  F$   s.t.  $|f|\in   F$.  We   set   $f^+=\max(0,f)$  and
  $f^-=\max(0,-f)$.  As $F$  is a  vector  space, we  get that  $f^+=(f+
  |f|)/2$ and $f^-=f^+-f$ belong to $F$. Notice that $|R_n(t)|\leq
  |R_n^+(t)|+ |R_n^-(t)|$, where $R_n^{\delta}(t)=m^{-n}M^*_{ \tau_n(t)}(f^{\delta})- \langle \mu , f^{\delta} \rangle m(m-1)^{-1}  W t$ for $\delta \in \{+,-\}$. So it is enough to prove the Corollary for
  $f$ non-negative. As $t\mapsto m^{-n}M^*_{ \tau_n(t)}(f)$ and
  $t\mapsto \langle \mu , f \rangle m(m-1)^{-1} t W$ are 
  non-decreasing and $R_n(0)=0$, we get that for $N\geq 1$, 
\[
\sup_{t\in [0,1]} |R_n(t)|\leq  \inv{N} \langle \mu, f  \rangle
m(m-1)^{-1} W
+\sum_{k=1}^N |R_n(k/N)| .
\]
Now, use that $W\in L^2$ and that $R_n(t)$ goes to $0$ in $L^2$ for all
$t \in [0,1]$ to get the result. 
\end{proof}

We have a version of Proposition \ref{prop:snt} and Corollary
\ref{cor:wllunif} for functions defined on $\bar S^3$. 

\begin{prop}
   \label{prop:wllS3tn}
  Let  $F$ satisfy  (i)-(vi),  $g\in \cb(\bar S^3)$ s.t. $P^* g$ and
  $P^* (g^2) $ exist
  and  belong to $F$. Let $t \in  [0,1]$. The
  sequence  $(m^{-n}M^*_{ \tau_n(t)}(g)$, $n  \in \N^*)$
  converges to $\langle \mu , P^* g \rangle m(m-1)^{-1} t W$ in $L^2$.
  
  Furthermore, if  $P^*  |g|$ and $P^* (g|g|)$
  also belong  to $F$  then $(\sup_{t\in [0,1]}  | R_n(t)|,n  \in \N^*)$
  converges  to 0  in $L^2$,  where $R_n(t)=  m^{-n}M^*_{ \tau_n(t)}(g)-
  \langle \mu , P^* g \rangle m(m-1)^{-1} Wt$, for $t\in [0,1]$.
\end{prop}

\begin{proof}
   The proof of the first part is similar to the proof of
   Theorem \ref{th:wlln4}. The proof of the second part is similar to
   the proof of Corollary \ref{cor:wllunif}.
\end{proof}

\section{Fluctuations}
\label{sec:tcl}

Recall \reff{eq:defMJ}.   For any real  $x \ge 0$, using  notations from
the previous  Section, we will  write $M_x^*(f)$ for  $M_{I^*_{\lfloor x
    \rfloor}}(f)$, with the  convention that $M_0^*(f)=0$, where $I^*_n$
is  defined by  \reff{eq:nfirst}.   We  shall prove  a central  limit
theorem  for the  sequence $(M_n^*(f),  n \ge  1)$, based  on martingale
theorems.

We  set $\mathcal{H}_n  =  \sigma(\Delta_{\tilde{\Pi}(k)}, 1  \le k  \le
n\wedge  |\T^*|) \vee  \sigma(\tilde{\Pi}(k),  1 \le  k  \le n+1)$  for
$n\geq 1$,  $\mathcal{H}_0 = \sigma(X_\emptyset)$  and $\ch=(\ch_n, n\in
\N)$  for  the  corresponding   filtration. With  the  convention  that
$X_{\partial_{\T}} =  \partial$, we notice  that $X_{\tilde{\Pi}(n+1)}$ is
$\mathcal{H}_n$-measurable.  Indeed,  given $(\tilde{\Pi}(k),1\leq k\leq
n+1)$, if $\tilde{\Pi}(n+1)\neq
\partial_\T$, we have  $\tilde{\Pi}(n+1)=\tilde{\Pi}(j)i$ for some $j\in
\{1,     \ldots,     n\}$     and     $i\in     \{0,1\}$,     and     as
$\Delta_{\tilde{\Pi}(j)}=(X_{\tilde{\Pi}(j)},        X_{\tilde{\Pi}(j)0},
X_{\tilde{\Pi}(j)1}) \in  \ch_n$, we deduce  that $X_{\tilde{\Pi}(n+1)}$
is  $\mathcal{H}_n$-measurable.   In  particular, as  $\{|\T^*|\geq
n+1\}\in  \ch_n$,  this implies  that  $\ind_{\{|\T^*|\geq n+1\}  }
\E[f(\Delta_{\tilde{\Pi}(n+1)}) | \mathcal{H}_n]=\ind_{\{|\T^*|\geq
  n+1\}  } P^*f(X_{\tilde{\Pi}(n+1)})$,  for any  $f \in  \cb(\bar S^3)$
such  that  $P^*f$  is  well  defined. If  in  addition  $P^*f=0$,  then
$(M^*_n(f), n\in \N)$ is an $\ch$-martingale.

We  shall first  recall  a slightly  weaker  version of  Theorem 4.3  from
\cite{r:cltmrct}   on   martingale    convergence.   (Theorem   4.3   from
\cite{r:cltmrct} is stated for filtrations which may vary with $n$.)

For $u\in \R^d$, we denote by $u'$ its transpose. Let $\ch=(\ch_i, i\in
\N)$ be a filtration. If $(D_i, i \in \N)$ is a sequence of vector
valued random
variables $\ch$-adapted  and such that 
$\E[D_{i+1}|\ch_i]=0$ for all $i \in \N$, then $(D_i, i \in \N)$ is called 
an $\ch$-martingale difference.

\begin{theo}[Theorem 4.3 from  \cite{r:cltmrct}]
\label{theo:root}
Let $\ch=(\ch_i,  i\in \N)$ be a  filtration. For all $n  \in \N^*$, let
$(D_{n,i}=(D_{n,i}^{(1)},  \ldots  ,D_{n,i}^{(d)}),  i\in  \N)  $  be  a
sequence  of  $\R^d$-valued   random  vectors  and  an  $\ch$-martingale
difference.   For each  $n\in \N$,  let $(\tau_n(t),  t\in [0,1])$  be a
non-decreasing  càdlàg function  s.t.  $\tau_n(t)$  is  a $\ch$-stopping
time  for  all   $t\in  [0,1]$.   Let  $(\Tau(t),  t\in   [0,1])$  be  a
$\R^{d\times d}$-valued continuous, possibly random, function. We assume
the following two conditions hold:
\begin{enumerate}
   \item  Convergence   of  the  timescales.  For all  $t\in [0,1]$, we
     have the following convergence in probability:
\[
 \sum_{i=1}^{\tau_n(t)}  \E\left[D_{n,i} (D_{n,i})'|\ch_{i-1}\right]
   \;\xrightarrow[n\rightarrow \infty ]{\P}\;\Tau(t).
\]
   \item Lindeberg condition.  For all $\varepsilon>0$, $1\leq \ell\leq d$,
     we have the following convergence in probability:
\[
 \sum_{i=1}^{\tau_n(1)} \E\left[(D_{n,i}^{(\ell)})^2\ind_{\{|D_{n,i}^{(\ell)}|>
       \varepsilon\}}|\ch_{i-1}\right]
   \;\xrightarrow[n\rightarrow \infty ]{\P}\; 0. 
\]
\end{enumerate}
Then $(\sum_{i=1}^{\lfloor \tau_n(\cdot) \rfloor}D_{n,i}, n\in \N^*)$ converges in
distribution to $B_\Tau$  in  the  Skorohod  space  $\D([0,1])^d$  of 
$\R^d$-valued  càdlàg functions defined on $[0,1]$, where, conditionally on $\Tau$,
$(B_\Tau(t), t\geq 0)$ is a Gaussian process with independent
increments and $B_\Tau(t)$ has zero mean and variance $\Tau(t)$. 

Furthermore the convergence is stable: if $(Y_n, n\in \N)$ converges in
probability to $Y$, then $((\sum_{i=1}^{\lfloor \tau_n(\cdot) \rfloor}D_{n,i} ,
Y_n) , n\in \N)$ converges in
distribution to
$(B_\Tau, Y)$, where $B_\Tau $ is conditionally on $(\Tau,Y)$
distributed as $B_\Tau $ conditionally on $\Tau$, and the
distribution of $(\Tau ,Y)$ is determined by the following
convergence 
\[
\Big(\sum_{i=1}^{\tau_n(\cdot)}  \E\left[D_{n,i} (D_{n,i})'|\ch_{i-1}\right], Y_n\Big)
   \;\xrightarrow[n\rightarrow \infty ]{\P}\;(\Tau,Y).
\]
\end{theo}

For the sake of simplicity, we will write $P^*h^k$ for $P^*(h^k)$, and
if $h=(h_1, \ldots, h_d)$ is an $\R^d$ valued function, we will write
$P^*h$ for $(P^* h_1, \ldots, P^* h_d)$ and $\langle \mu,h  \rangle$ for
$(\langle \mu, h_1 \rangle, \ldots, \langle \mu, h_d \rangle)$.

\begin{theorem}
\label{theo:stable}
Let  $(X_i,  i\in  \T^*)$  be  a  super-critical  spatially  homogeneous
$P^*$-BMC on  a GW tree and  $\tau_n$ be defined  by \reff{eq:snt}.  Let
$F$ satisfy (i)-(vi).  Let $d\geq  1$, $d'\geq 1$, $f=(f_1, \ldots, f_d)
\in  \cb(\bar S^3)^d,  g=(g_1, \ldots,  g_{d'}) \in  \cb(\bar S^3)^{d'}$
such  that  $P^* f_\ell^k$,  exist  and belong  to  $F$  for all  $1\leq
\ell\leq  d$ and  $1\leq k\leq  4$, $P^*  g_\ell$, $P^*  |g_\ell|$, $P^*
g_\ell^2$  and $P^*  g_\ell|g_\ell|$ exist  and  belong to  $F$ for  all
$1\leq  \ell\leq d'$. Let  $\Sigma$ be  a square  root of  the symmetric
positive  matrix $m(m-1)^{-1}\langle  \mu,  P^* (ff')  -(P^* f)(P^*  f)'
\rangle$ and $\gamma=m(m-1)^{-1} \langle \mu ,P^* g \rangle$.

  Then,  the sequence  $(m^{-n/2}  M^*_{\tau_n(\cdot)} (f-P^*f),  m^{-n}
  M^*_{\tau_n(\cdot)}(g))$  converges in  distribution  in the  Skorohod
  space $\D([0,1], \R^{d+d'})$  of $\R^{d+d'}\!$-valued càdlàg functions
  defined  on $[0,1]$,  to the  process  $(\Sigma \sqrt{W}  B, \gamma  W
  h_0)$, where  $B$ is a $d$-dimensional Brownian  motion independent of
  $W$,  defined by \reff{eq:defW},  and $h_0$  is the  identity function
  $t\mapsto t$.
\end{theorem}

\begin{proof}
  Notice   that  $\tau_n$   defined   by  \reff{eq:snt} is a non-decreasing
  continuous function s.t.  $\tau_n(t)$ is a $\ch$-stopping time for all
  $t\in [0,1]$. We set for all $n, i \in  \N^*$, 
\[
D_{n,i}=m^{-n/2} \left( f(\Delta_{\tilde \Pi(i)}) -
  P^*f(X_{\tilde \Pi (i)})\right)\ind_{\{i\leq |\T^*|\}},
\]
so that $(D_{n,i}, i\in  \N)$ is an $\ch$-martingale difference. Notice the matrix $\langle \mu, P^* ff' -(P^* f)(P^*
  f)' \rangle$ is indeed symmetric and positive, so that $\Sigma$ is well
  defined. 

Notice that 
\[
\E\left[D_{n,i}(D_{n,i})'|\ch_{i-1}\right]= m^{-n} \left(P^* (ff')(X_{\tilde
  \Pi(i)}) -  (P^* f)(X_{\tilde
  \Pi(i)})(P^*  f)'(X_{\tilde  \Pi(i)})\right)   \ind_{\{i\leq |\T^*|\}}.
\]
The convergence  of the timescales condition  of Theorem \ref{theo:root}
with $\Tau(t)=\Sigma^2 W t$, is then a direct application of Proposition
\ref{prop:snt}

For $1\leq \ell\leq d$, we have 
\[
\E\left[(D_{n,i}^{(\ell)})^2\ind_{\{|D_{n,i}^{(\ell)}|>
       \varepsilon\}}|\ch_{i-1}\right]
\leq  \varepsilon^{-2} \E\left[(D_{n,i}^{(\ell)})^4|\ch_{i-1}\right]=
\varepsilon^{-2}  m^{-2n} P^* (f_\ell - P^* 
   f_\ell)^4 (X_{\tilde 
  \Pi(i)})\ind_{\{i\leq |\T^*|\}}. 
\]
The Lindeberg condition  of Theorem \ref{theo:root}
 is then  a direct  application of Proposition \ref{prop:snt}

Notice the  second part  of Proposition \ref{prop:wllS3tn}.  implies the
convergence of $Y_n=m^{-n} M^*_{\tau_n(\cdot)}(g)$  to $\gamma W h_0$ in
probability in  the Skorohod space. 
We then deduce the result from Theorem \ref{theo:root}. 
\end{proof}

The following result is an immediate consequence of  Theorem
\ref{theo:stable}.  

\begin{cor}
  Let  $(X_i,  i\in \T^*)$  be  a  super-critical spatially  homogeneous
  $P^*$-BMC  on a  GW  tree. Let  $F$  satisfy (i)-(vi).   Let  $ f  \in
  \cb(\bar S^3)$ such  that $P^* f^k$ exists and belongs  to $F$ for all
  $1\leq  k\leq   4$.   Let  $\sigma^2=\langle  \mu,  P^*   f^2  -  (P^*
  f)^2\rangle $.

Then we have the following convergence in distribution: 
\[
\ind_{\{|\G^*_n|>0\}} |\T^*_n|^{-1/2} \sum_{i\in \T^*_n}
f(\Delta_i) - P^* f(X_i) 
   \;\xrightarrow[n\rightarrow \infty ]{\text{(d)}}\;
\ind_{\{W\neq 0\}} \sigma G,
\]
where $G$ is a Gaussian random variable with
mean zero and variance~$1$ independent of $W$, which is defined by
\reff{eq:defW}.    
\end{cor}

\begin{proof}
  Notice  that   $  \sum_{i\in  \T^*_n}  f(\Delta_i)  -   P^*  f(X_i)  =
  M^*_{\tau_n(1)}(f)-M^*_{\tau_n(1)}(P^*f)$,         $         |\T^*_n|=
  M^*_{\tau_n(1)}(\ind)$ and that $\ind_{\{|\G^*_n|>0\}}$ converges a.s.
  to $\ind_{\{W\neq 0\}}$. Then, to conclude, use the stable convergence
  of Theorem \ref{theo:stable}, and the  fact that the marginals at time
  1 converge since the limit is continuous.
\end{proof}

The next result gives that the fluctuations over each generation are
asymptotically independent. 

\begin{cor}
     Let  $(X_i,  i\in \T^*)$  be  a  super-critical spatially  homogeneous
  $P^*$-BMC  on a  GW  tree. Let $F$  satisfy (i)-(vi).  Let  $d\geq 1$,   $f_1, \ldots,
  f_d  \in  \cb(\bar  S^3)$ such  that $P^*  f_\ell^k$
    exist and
  belong to $F$ for all $1\leq  \ell\leq d$ and $1\leq k\leq 4$. 
Let $\sigma_\ell^2=\langle \mu,   P^* f_\ell^2 - (P^* f_\ell)^2\rangle $
for  $1\leq  \ell\leq d$. 

We set for $f\in \cb(\bar S^3)$
\[
N_n(f)=|\G^*_n|^{-1/2} (M_{\G^*_n}(f-P^*f)).
\]
Then we have the following convergence in distribution: 
\[
\left(N_n(f_1), \ldots, N_{n-d+1}(f_d)\right)\ind_{\{|\G^*_n|>0\}}
   \;\xrightarrow[n\rightarrow \infty ]{\text{(d)}}\;
\ind_{\{W\neq 0\}} (\sigma_1 G_1, \ldots, \sigma_d G_d),
\]
where $G_1, \ldots,  G_d$ are independent Gaussian random variables with
mean zero and variance~$1$ and are independent of $W$, which is defined by
\reff{eq:defW}.    
\end{cor}

\begin{proof}
Notice that for $n> k\geq 0$, 
\[
N_{n-k}(f)\ind_{\{|\G^*_n|>0\}} =\frac{
  M^*_{\tau_n(m^{-k} )} (f-P^*f) -  M^*_{\tau_n(m^{-k-1} )} (f-P^*f)}
{ \sqrt{  M^*_{\tau_n(m^{-k} )} (\ind ) - M^*_{\tau_n(m^{-k-1} )} (\ind
    )}} \ind_{\{|\G^*_n|>0\}} 
\]
and $\ind_{\{|\G^*_n|>0\}}$  converges a.s. to  $\ind_{\{W\neq 0\}}$. To
conclude, use  the stable  convergence of Theorem  \ref{theo:stable} and
that the increments of the Brownian motion are independent.
\end{proof}

The extension of the two previous Corollaries to vector-valued function
can be proved in a very similar way. 

\section{Estimation and tests for the asymmetric auto-regressive model}
\label{sec:test}

We  consider  the  asymmetric   auto-regressive  model  given  in  Section
\ref{ssec:stat}.  Notice that the  process $(X_i,  i\in \T)$  defined in
Section \ref{ssec:stat}  with the convention that  $X_i=\partial$ if the
cell  $i$  is  dead  or  non  existing  is  a   spatially
homogeneous BMC on a GW tree. We shall assume it is super-critical, that
is $2p_{1,0}+p_1+p_0>1$. 

We compute the maximum likelihood estimator (MLE) 
\[
\hat \theta^n = (\hat{\alpha}_0^n, \hat{\beta}_0^n, \hat{\alpha}_1^n,
\hat{\beta}_1^n, \hat{\alpha}_0'^n, \hat{\beta}_0'^n, \hat{\alpha}_1'^n,
\hat{\beta}_1'^n, \hat{p}_{1,0}^n, \hat{p}_0^n, \hat{p}_1^n)
\]
of $\theta$ given by \reff{eq:theta}  and $\kappa^n=(\hat \sigma^n,
\hat \rho  ^n, \hat \sigma  _0^n, \hat \sigma_1^n)$  of $\kappa=(\sigma,
\rho,  \sigma _0,  \sigma_1)$, based  on  the observation  of a  sub-tree
$\T^*_{n+1}$. Let $\T_n^{1,0}$ be the  set of cells in $\T_n^*$ with two
living  daughters, $\T_n^0$  (resp. $\T_n^1$)  be  the set  of cells  of
$\T_n^*$ with only the new (resp.  old) pole daughter alive:
\[
\T_n^{1,0} = \{i \in \T^*_n, \Delta_i \in S^3\},\, \T_n^0 = \{i \in
\T^*_n, \Delta_i \in S^2 \times \{\partial\}\}      \text{ and }
\T_n^1 = \{i \in 
\T^*_n, \Delta_i \in S \times \{\partial\} \times S\}.
\]
It is elementary to get that   for $\delta
\in \{0,1\}$,
\begin{align}
\label{eq:hat-a}
 \hat{\alpha}_{\delta}^n & =  \frac{\displaystyle{|\T_n^{1,0}|^{-1}
     \sum_{i \in \T_n^{1,0}}X_iX_{i\delta} - (|\T_n^{1,0}|^{-1} \sum_{i
       \in \T_n^{1,0}}X_i)(|\T_n^{1,0}|^{-1}\sum_{i \in \T_n^{1,0}}
     X_{i\delta})}}{\displaystyle{|\T_n^{1,0}|^{-1} \sum_{i \in
       \T_n^{1,0}}X_i^2 - (|\T_n^{1,0}|^{-1}\sum_{i \in
       \T_n^{1,0}}X_i)^2}}, \\ 
\label{eq:hat-b}
 \hat{\beta}_{\delta}^n & =  \displaystyle{|\T_n^{1,0}|^{-1} \sum_{i \in
     \T_n^{1,0}} X_{i\delta} - \hat{\alpha}_{\delta}^n |\T_n^{1,0}|^{-1}
   \sum_{i \in \T_n^{1,0}} X_i},\\ 
 \nonumber
\hat{\alpha}_{\delta}'^n & =  \frac{\displaystyle{|\T_n^{\delta}|^{-1}
     \sum_{i \in \T_n^{\delta}}X_iX_{i\delta} - (|\T_n^{\delta}|^{-1}
     \sum_{i \in \T_n^{\delta}}X_i)(|\T_n^{\delta}|^{-1}\sum_{i \in
       \T_n^{\delta}} X_{i\delta})}}{\displaystyle{|\T_n^{\delta}|^{-1}
     \sum_{i \in \T_n^{\delta}}X_i^2 - (|\T_n^{\delta}|^{-1}\sum_{i \in
       \T_n^{\delta}}X_i)^2}}, \\ 
 \nonumber
 \hat{\beta}_{\delta}'^n & =   \displaystyle{|\T_n^{\delta}|^{-1}
   \sum_{i \in \T_n^{\delta}} X_{i\delta} - \hat{\alpha}_{\delta}'^n
   |\T_n^{\delta}|^{-1} \sum_{i \in \T_n^{\delta}} X_i},\\ 
 \hat{p}_{1,0}^n & =  \displaystyle{\frac{|\T_n^{1,0}|}{|\T_n^*|}},\quad
 \hat{p}_{\delta}^n  =  \displaystyle{\frac{|\T_n^{\delta}|}{|\T_n^*|}},
\end{align}
and 
\[
   (\hat\sigma^n)^2= \inv{2|\T^{1,0}_n|} \sum_{i\in\T^{1,0}_n} (\hat
   \varepsilon^2_{i0} +\hat
   \varepsilon^2_{i1})  ,\quad
   \hat\rho^n=\inv{(\hat\sigma^n)^2 |\T^{1,0}_n|} \sum_{i\in\T^{1,0}_n} \hat
   \varepsilon_{i0} \hat
   \varepsilon_{i1} ,\quad\text{and}\quad
   (\hat\sigma^n_\delta)^2=\inv{|\T^{\delta}_n|} \sum_{i\in\T^{\delta}_n} \hat
   \varepsilon'^2_{i\delta} .
\]
The residues are
\[
   \hat \varepsilon_{i\delta}=X_{i\delta} - \hat\alpha^n_\delta X_i-
   \hat\beta^n_\delta  \quad \text{for $i\in
     \T^{1,0}_n$},\quad\text{and}\quad 
   \hat \varepsilon'_{i\delta}=X_{i\delta} - \hat\alpha'^n_\delta X_i-
   \hat\beta'^n_\delta  \quad \text{for $i\in \T^{\delta}_n$}.
\]

Notice  that  those  MLE  are  based  on  polynomial  functions  of  the
observations. In order to use the results of Sections \ref{sec:wlln} and
\ref{sec:tcl}, we first show that the set of continuous and polynomially
growing functions  satisfies properties $(i)$ to $(v)$ of Section
\ref{sec:Yn}. The set of continuous and polynomially
growing functions $\pol$ is defined as the set of continuous real
functions defined on $\R$,  s.t. there exists $m\geq 0$ and $c\geq 0$ and
for all $x\in \R$, $|f(x)|\leq c(1+|x|^m)$. 
It is easy to check that $\pol$ satisfies conditions $(i)$-$(iii)$. To
check properties $(iv)$ and $(v)$, we notice that the auxiliary Markov
chain $Y=(Y_n, n\in \N)$ can be written in 
the following way:
$$Y_{n+1} = a_{n+1}Y_n + b_{n+1},$$
with $b_n= b_n'  + s_n e_n$, where  $((a_n, b_n', s_n), n \ge  1)$ is a
sequence of independent  identically distributed random variables, whose
common distribution is given by, for $\delta \in \{0,1\}$,
\begin{equation}
\label{eq:loimu}
\P(a_1=\alpha_{\delta}, b_1'=\beta_{\delta}, s_1=\sigma)=
\frac{p_{1,0}}{m} \quad \mbox{and} \quad \P(a_1=\alpha_{\delta}',
b_1'=\beta_{\delta}', s_1=\sigma_{\delta})= \frac{p_{\delta}}{m},
\end{equation}
$(e_n  ,  n \ge  1)$  is a  sequence  of  independent $\cn(0,1)$  random
variables, and is independent of $((a_n, b_n', s_n), n \ge 1)$, and both
sequences are independent of $Y_0$.  Notice that $Y_n$ is distributed as
$Z_n=a_1a_2   \cdots  a_{n-1}a_n  Y_0   +  \sum_{k=1}^n   a_1a_2  \cdots
a_{k-1}b_k$. Since  $|a_k|\leq \max(|\alpha_0|, |\alpha_1|, |\alpha'_0|,
|\alpha'_1|)<1$ for  all $k\in  \N^*$, we get  that the  sequence $(Z_n,
n\in  \N)$  converges  a.s. to  a  limit  $Z$.   This implies  that  $Y$
converges in  distribution to $Z$.  Following  the proof of  Lemma 26 in
\cite{g:ltbmcadca}, we  get that  $\pol$ fulfills properties  $(iv)$ and
$(v)$, with $\mu$ the distribution of $Z$.

\begin{prop}\label{prop:CVqk}
  Assume that  the distribution of the  ancestor $X_{\emptyset}$
  has   finite  moments   of  all   orders.   Then  $(\ind_{\{|\G_n^*|>0
    \}}\hat{\theta}^n,    n    \ge    1)$    and    $(\ind_{\{|\G_n^*|>0
    \}}\hat{\kappa}^n, n  \ge 1)$ converges  in probability respectively
  to $\ind_{\{W \neq 0 \}}\theta$ and $\ind_{\{W \neq 0 \}}\kappa$,
  where $W$ is defined by \reff{eq:defW}.
\end{prop}

\begin{proof}
  The  hypothesis  on the  distribution  of  $X_\emptyset$ implies  that
  $\pol$ fulfills  $(vi)$. The  result is then  a direct  consequence of
  Theorem \ref{th:wlln4}.
\end{proof}

\begin{rem}
\label{rem:LFGN}
Using   similar   arguments   as   in   Proposition   30   and   34   of
\cite{g:ltbmcadca}, it  is easy to  deduce from \reff{eq:ml2mg}  and the
proofs of Theorem  14 and Proposition 28 of  \cite{g:ltbmcadca} that the
convergences in Proposition \ref{prop:CVqk}  hold a.s., that is the MLEs
$\hat{\theta}^n$ and $\hat\kappa^n$ are strongly consistent.
\end{rem}

From  the   definition  of  $Z_n$,   we  deduce  that   in  distribution
$Z\overset{(d)}{=} a_1Z'+b_1$,  where $Z'$ is distributed as  $Z$ and is
independent of $(a_1, b_1)$ (see \reff{eq:loimu} for the distribution of
$(a_1, b_1)$). This equality in distribution entails that
\begin{equation}
   \label{eq:defm1-2}
\mu_1 = \E[Z]=\frac{\bar \beta}{1-\bar \alpha}
\quad \mbox{and} \quad \mu_2 =\E[Z^2]=
\frac{2 \overline{\alpha \beta} \bar \beta /(1-\bar
  \alpha)+\overline{\beta^2}+\overline{\sigma^2}}{1-\overline{\alpha^2}},
\end{equation}
where $\bar \alpha = \E[a_1]$, $\overline{\alpha^2} = \E[a_1^2]$, $ \bar
\beta =  \E[b_1]$, $ \overline{\beta^2}  = \E[b_1^2]$, $\overline{\alpha
  \beta}  =  \E[a_1b_1]$  and  $\overline{\sigma^2}  =  \E[s_1^2]$. 

We can now state one of the main result of this section.
\begin{prop}
\label{prop:lan}
Assume that the distribution  of the ancestor $X_{\emptyset}$ has finite
moments    of     all    orders.     Then     $\ind_{\{|\G_n^*|>0    \}}
|\T_n^*|^{1/2}(\hat{\theta}^n-\theta)$  converges in  law  to $\ind_{\{W
  \neq  0  \}}  G_{11}$,  where  $G_{11}$ is  a  11-dimensional  vector,
independent of  $W$ defined by \reff{eq:defW},  with law $\mathcal{N}(0,
\Sigma)$ where
$$\Sigma =\left(
\begin{array}{cccccccc}
\sigma^2 K/p_{1,0} & \rho \sigma^2 K/p_{1,0} & 0 & 0 & 0 \\
\rho \sigma^2 K/p_{1,0} & \sigma^2 K/p_{1,0} & 0 & 0 & 0\\
0 & 0 & \sigma_0^2 K/p_0 & 0 & 0 \\
0 & 0 & 0 &  \sigma_1^2 K/p_1 & 0 \\
 0 & 0 & 0 & 0 & \Gamma \\
\end{array}
\right) \quad \mbox{with}$$
$$K= (\mu_2 -\mu_1^2)^{-1}\left(
\begin{array}{cc}
 1 & - \mu_1 \\
-\mu_1  & \mu_2
\end{array}
\right) \quad \mbox{and} \quad \Gamma=\left(
\begin{array}{ccc}
 p_{1,0}(1-p_{1,0}) & -p_0 p_{1,0} & -p_1p_{1,0} \\
-p_0 p_{1,0} & p_0(1-p_0) & -p_0p_1 \\
 -p_1p_{1,0} & -p_0p_1 & p_1(1-p_1) 
\end{array}
\right).
$$
\end{prop}

The   proof   of    Proposition   \ref{prop:lan}   relies   on   Theorem
\ref{theo:stable}   and  mimics   the   proof  of   Proposition  33   of
\cite{g:ltbmcadca}. It is left to the reader. 

\begin{rem}
\label{rem:cv}
Proposition \ref{prop:lan}  deals with  the asymptotic normality  of the
MLE   of   $\theta$  based   on   the   observation   of  the   sub-tree
$\T^*_{n+1}$.  If  $L(X_i,  i   \in  \T^*_{n+1},  \theta)$  denotes  the
corresponding log-likelihood  function for $\theta$, the 
Fisher information, say $I_{n+1}$, is given by 
\[
I_{n+1}=-\E\left[\frac{\partial^2L(X_i, i \in \T^*_{n+1},
    \theta)}{\partial \theta \partial \theta'}\right].
\]
Using Theorem \ref{th:wlln4}, one can check that $\lim_{n\rightarrow
  \infty } I_{n+1}/\E[|\T^*_{n+1}|]= \Sigma^{-1}$. This is  the analogue
of the well-known asymptotic efficiency of the MLE for parametric sample
of i.i.d. random variables. 
\end{rem}

  Let   $\theta_{1,0}$   (resp.    $\hat{\theta}_{1,0}^n$)   stand   for
  $(\alpha_0,  \beta_0, \alpha_1,  \beta_1)$  (resp. $(\hat{\alpha}_0^n,
  \hat{\beta}_0^n,  \hat{\alpha}_1^n,  \hat{\beta}_1^n)$).   

\begin{rem}
  Proposition  \ref{prop:lan}  is quite  similar  to  Proposition 33  in
  \cite{g:ltbmcadca}.  One of the main differences comes from the factor
  $p_{1,0}^{-1}$ in front of the matrix $K$ in the asymptotic covariance
  matrix     for    the     estimation     of    $\theta_{1,0}$     with
  $\hat{\theta}_{1,0}^n$. As  a matter of  fact, this factor  comes from
  the normalization  by $|\T_n^*|^{1/2}$, number  of living cells  up to
  generation $n$, whereas  this estimation is related to  the cells with
  two  living   daughters,  which   would  induce  a   normalization  by
  $|\T_n^{1,0}|^{1/2}$.        Since       $\ind_{\{|\T_n^*|>0       \}}
  |\T_n^{1,0}|/|\T_n^*|$  converges in probability  to $p_{1,0}\ind_{\{W
    \neq  0\}}$,   such  a  normalization  would   suppress  the  factor
  $p_{1,0}^{-1}$, see the following Corollary.
\end{rem}

\begin{cor}
\label{cor:4}
Assume that the distribution of the ancestor $X_{\emptyset}$ has
finite moments of all orders. Then $\ind_{\{|\G_n^*|>0 \}}
|\T_n^{1,0}|^{1/2}(\hat{\theta}_{1,0}^n-\theta_{1,0})$ converges in law
to $\ind_{\{W \neq 0 \}} G_4$, where $G_4$ is a 4-dimensional vector,
independent of $W$ defined by \reff{eq:defW}, with law $\mathcal{N}(0,
\Sigma')$ where  
$$\Sigma' =\sigma^2 \left(
\begin{array}{cc}
K & \rho K\\
\rho K & K
\end{array}
\right) \quad \mbox{with} \quad K= (\mu_2 -\mu_1^2)^{-1}\left(
\begin{array}{cc}
 1 & - \mu_1 \\
-\mu_1  & \mu_2
\end{array}
\right).$$
\end{cor}

This   result   is   formally    the   same   as   Proposition   33   of
\cite{g:ltbmcadca}, but  one should notice that $\mu_1$  and $\mu_2$ are
not defined the same way  as in \cite{g:ltbmcadca}, since here they also
depends  on  the  parameters  concerning  cells with  dead  sisters  see
equations   \reff{eq:dead-0},   \reff{eq:dead-1}, \reff{eq:loimu} and
\reff{eq:defm1-2}.

In order to detect cellular aging, see \cite{g:ltbmcadca} in the case of
no  death ($m=2$), we  consider the  null hypothesis  $H_0=\{(\alpha_0 ,
\beta_0)=(\alpha_1, \beta_1)\}$,  which corresponds to no  aging and its
alternative        $H_1=\{(\alpha_0       ,       \beta_0)\neq(\alpha_1,
\beta_1)\}$.  Notice that  $\theta\mapsto  \mu_1(\theta)$ and  $(\theta,
\kappa)\mapsto  \mu_2(\theta,\kappa)$  given  by  \reff{eq:defm1-2}  are
continuous  functions  defined  respectively on  $\Theta=((-1,1)  \times
\R)^4    \times   ([0,1]^3\setminus   \{0,0,0\})$    and   $\Theta\times
]0,+\infty[^3$.   We  set  $\hat{\mu}_1^n =  \mu_1(\hat{\theta}^n)$  and
$\hat{\mu}_2^n  =  \mu_2(\hat{\theta}^n,\hat{\kappa}^n )$.   Proposition
\ref{prop:test}  allows to  build a  test for  $H_0$ against  $H_1$. Its
proof, which is left to the  reader, follows the proof of Proposition 35
of \cite{g:ltbmcadca}  and uses Corollary \ref{cor:4}, the  value of the
extinction probability  $\eta=\P(W=0)=1- \frac{m-1}{p_{1,0}}$, where $W$
is defined by \reff{eq:defW} and Remark \ref{rem:LFGN}.

\begin{prop}
\label{prop:test}
Let $U$ and $V$ be two independent random variables, with $U$ distributed
as a  $\chi^2$  with two degrees of
freedom and $V$ a Bernoulli random variable with parameter $1-\eta$. 

Assume that the distribution of the ancestor $X_{\emptyset}$ has
finite moments of all orders and define the test statistic
\[
\zeta_n = \frac{|\T_n^{1,0}|}{2(\hat\sigma^n)^2(1-\hat\rho^n)}
\left\{(\hat{\alpha}_0^n-\hat{\alpha}_1^n)^2 (\hat{\mu}_2^n -
  (\hat{\mu}_1^n)^2) + \left(
    (\hat{\alpha}_0^n-\hat{\alpha}_1^n)\hat{\mu}_1^n +
    \hat{\beta}_0^n-\hat{\beta}_1^n \right)^2 \right\}.
\]
Then,  the statistics $\ind_{\{|\G_n^*|>0  \}} \zeta_n$  converges under
$H_0$ in distribution to $UV$, and  under $H_1$ a.s. to $0$ on $\{V=0\}$
and $+\infty $ on $\{V=1\}$.
\end{prop}

\begin{rem}
   \label{rem:conserve}
   Let us assume that:
\begin{itemize}
   \item Death occurs, that is $m\in (1,2)$.
   \item There  is no difference for the
   marginal  distribution of a daughter according to her sister is dead
   or alive; that is   $\alpha'_\delta=\alpha_\delta$  and
   $\beta'_\delta=\beta_\delta$ for  $\delta\in  \{0,1\}$. 
   \item For simplicity, the death
   probability is symmetric, that is  $p_0=p_1$.
\end{itemize}
If  one uses the
statistics given by Proposition  33 in  \cite{g:ltbmcadca} with all the
available data, that is if one  uses 
\begin{itemize}
   \item Formula \reff{eq:hat-a} and \reff{eq:hat-b}
with  $\T_n^{1,0}$ replaced by 
$\T_n^{1,0}\cup \T_n^\delta$;
   \item The variance estimator: 
\[
(\hat\sigma^n)^2=\inv{|\T^*_{n+1}|-1}\Big(\sum_{i \in \T_n^{1,0}}
(\hat\varepsilon^2_{i0} + \hat\varepsilon^2_{i1}) + \sum_{i\in\T_n^0}
\hat\varepsilon^2_{i0} + \sum_{i\in\T_n^1} \hat\varepsilon^2_{i1} 
\Big);
\]
(Notice that we divide by $|\T^*_{n+1}| -1$ as this is equal to the total
number of data: $2|\T_n^{1,0}|+|\T_n^0|+|\T_n^1|$.)
\item    Keep    the     same    estimation    of    the    correlation:
  $\hat\rho^n=\inv{(\hat\sigma^n)^2  |\T^{1,0}_n|} \sum_{i\in\T^{1,0}_n}
  \hat \varepsilon_{i0} \hat \varepsilon_{i1}$;
\end{itemize}
then one check that, as $n$ goes to infinity, 
 $\ind_{\{|\G_n^*|>0 \}} |\T_n^*|^{1/2}(\hat\theta^n-\theta)$ converges
 in distribution to  $\ind_{\{W \neq 0 \}} G$, where $G$ is a centered
 Gaussian vector with covariance matrix 
\[
\sigma^2(p_{1,0}+p_1)^{-1}\left(\begin{array}{cc}
         K & \rho p_{1,0}(p_{1,0}+p_1)^{-1} K\\
         \rho p_{1,0}(p_{1,0}+p_1)^{-1} K & K
        \end{array}
  \right),
\]
where $K$ is  as in Proposition \ref{prop:lan}; and  $G$ is independent
of $W$, which is defined by \reff{eq:defW}.
Then, it is not difficult to check that the statistics proposed by Guyon
in Proposition 35 of \cite{g:ltbmcadca}, converges under $H_0$ towards
$cUV$, with $U$ and $V$ as in Proposition \ref{prop:test} and 
\[
c= \frac{(p_{1,0}+p_1)^{-1}(1-\rho
  p_{1,0}(p_{1,0}+p_1)^{-1})}{(1-\rho)}.
\]
As $\rho \in [-1,1]$,  $p_{1,0}+p_1 > 1/2$ (because $m>1$ and $p_0=p_1$)
and $2p_1+p_{1,0} \le 1$, one can check that $c>1$. 
In particular, using the test statistic designed for cells with no
death to data of cells with death leads to a non-conservative test. 
\end{rem}

\newcommand{\sortnoop}[1]{}

\end{document}